\documentclass[11pt]{amsart}

\author[C.~Sanna]{Carlo Sanna$^\dagger$}
\thanks{$\dagger\,$C.~Sanna is a member of GNSAGA of INdAM and of CrypTO, the group of Cryptography and Number~Theory of Politecnico di Torino}
\address{\parbox{\linewidth}{
Department of Mathematical Sciences, Politecnico di Torino\\
Corso Duca degli Abruzzi 24, 10129 Torino, Italy\\[-8pt]}}
\email{carlo.sanna@polito.it}


\title{Counting contiguous superregular $4 \times 4$ matrices}

\usepackage{amsmath}
\usepackage{amssymb}
\usepackage{amsthm}
\usepackage{bm}
\usepackage[left=1.15in, right=1.15in, top=.72in, bottom=.72in]{geometry}
\usepackage[colorlinks=true]{hyperref}
\usepackage{comment}
\usepackage{cleveref}
\usepackage{enumitem}
\setlist[enumerate]{label=(\roman*),labelindent=1em,itemsep=0.5em,topsep=0.5em}

\newtheorem{theorem}{Theorem}[section]

\newtheorem{lemma}[theorem]{Lemma}
\theoremstyle{remark}

\allowdisplaybreaks
\begin{document}

\begin{abstract}
    This short paper has two goals.
    First, explaining a simple procedure (which is essentially folklore) that, sometimes, makes it possible to obtain a formula for the number of solutions to a system of multivariate polynomial inequalities over a finite field.
    Second, applying that procedure to prove a formula for the number of contiguous superregular $4 \times 4$ matrices over a finite field.
    The formula was previously conjectured by Appuswamy, Bazzani, Connelly, Ekaireb, Congero, and Zeger [Probability of super-regular matrices and MDS codes over finite fields, arXiv:2603.20983].
    In addition, the same procedure is used to provide formulas for the number of contiguous superregular $3 \times 4$, $3 \times 5$, and $3 \times 6$ matrices over a finite field.
\end{abstract}

\maketitle

\section{Introduction}

A \emph{submatrix} of a matrix $\bm{M}$ is a matrix obtained by deleting possibly some, but not all, of the rows of $\bm{M}$ and possibly some, but not all, of the columns of $\bm{M}$.
A \emph{contiguous submatrix} has the additional property that the nondeleted rows and columns are contiguous to each others.
A matrix is \emph{superregular} if all its square submatrices are invertible, and it is \emph{contiguous superregular} if all its square contiguous submatrices are invertible.
Superregular matrices over finite fields play a significant role in the theory of error-correcting codes.
Indeed, the important \emph{maximum distance separable} (MDS) codes are exactly those codes whose generator matrix can be put in the form $(\bm{I} \mid \bm{S})$, where $\bm{I}$ is an identity matrix and $\bm{S}$ is a superregular matrix over a finite field.
(For an introduction to the theory of error-correcting codes, see the book by MacWilliams and Sloane~\cite{MR465509}.
In particular, the aforementioned result is Theorem~8 of Chapter~11.)

Let $\mathbb{F}_q$ denote a finite field of cardinality $q$.
For each positive integer $k$, let $C_k^{\text{CSR}}(q)$ denote the number of $k \times k$ contiguous superregular matrices over $\mathbb{F}_q$.
Note that $C_1^{\text{CSR}}(q) = q - 1$ and
\begin{equation*}
    C_2^{\text{CSR}}(q) = (q - 1)^3(q - 2) .
\end{equation*}
Appuswamy, Bazzani, Connelly, Ekaireb, Congero, and Zeger~\cite[Theorem 4.4(b)]{Appuswamy2026} proved that
\begin{equation*}
    C_3^{\text{CSR}}(q) = (q - 1)^5 (q - 2) (q - 3) (q^2 - 4q + 5) .
\end{equation*}
Furthermore, they conjectured \cite[Conjecture 4.10]{Appuswamy2026} that
\begin{align}\label{equ:CSR4x4}
    C_4^{\text{CSR}}(q) &= (q - 1)^7 (q - 2) (q - 3) \\
    &\hspace{2em}\cdot (q^7 - 18q^6 + 143q^5 - 654q^4 + 1874q^3 - 3400q^2 + 3671q - 1855) . \nonumber
\end{align}
In particular, they performed an exhaustive computer search that verified Eq.~\eqref{equ:CSR4x4} for all prime powers $q$ up to $67$ \cite[Table~1]{Appuswamy2026}.

The main result of this paper is a proof of Eq.~\eqref{equ:CSR4x4}.
Unfortunately, the proof provides no insight into the structure of contiguous superregular matrices.
Indeed, the core of the proof consists of running a computer implementation of a procedure that tries to determine the number of solutions to certain systems of polynomial inequalities.
The procedure ``tries'' because, in general, it either outputs the correct result or aborts.
Luckily, when given as input some inequalities corresponding to the condition of being a contiguous superregular $4 \times 4$ matrix, the procedure terminates and its output implies Eq.~\eqref{equ:CSR4x4}.
The procedure is quite simple and essentially folklore.
The source code of an implementation in \textsf{SageMath}~\cite{SageMath} is available in a public repository~\cite{Repo}.

The rest of the paper is structured as follows.
Section~\ref{sec:poly-ineqs} introduces the aforementioned problem of determining the number of solutions to certain systems of polynomial inequalities; Section~\ref{sec:procedure} provides a high-level explaination of the procedure to solve it; Section~\ref{sec:CSR4x4} gives the proof of Eq.~\eqref{equ:CSR4x4}; and Section~\ref{sec:CSR3x4-5-6} provides the formulas for the number of contiguous superregular $3 \times 4$, $3 \times 5$, and $3 \times 6$ matrices over a finite field.

\section{Counting solutions to polynomial inequalities}\label{sec:poly-ineqs}

Suppose that a tuple of multivariate polynomials $\bm{f} = f_1, \dots, f_m \in \mathbb{Z}[\bm{x}]$ in the formal variables $\bm{x} = x_1, \dots, x_n$ is given as input.
For every prime power $q$, define
\begin{equation*}
    N(\bm{f} ; q) := \big|\big\{\bm{x} \in \mathbb{F}_q^n : {\textstyle\bigwedge}_{i = 1}^m \big(f_i(\bm{x}) \neq 0 \big) \big\}\big| .
\end{equation*}
The problem is to determine a formula for $N(\bm{f} ; q)$ in terms of $q$ and possibly some simple functions of $q$.
In principle, there is no loss of generality in assuming that $m = 1$, since
\begin{equation}\label{equ:poly-ineqs:1}
    N(\bm{f}; q) = N(f_1 \cdots f_m ; q) .
\end{equation}
However, in practice, using an arbitrary $m$ is convenient because it makes possible to work with the lower-degree polynomials $f_1, \dots, f_m$ rather than the higher-degree product $f_1 \cdots f_m$.
Furthermore, note that the computation of $N(\bm{f}; q)$ is equivalent to the computation of the number of solutions to a polynomial equality.
Indeed, from Eq.~\eqref{equ:poly-ineqs:1} it follows that
\begin{equation*}
    N(\bm{f}; q) := q^n - \big|\big\{\bm{x} \in \mathbb{F}_q^n : f_1(\bm{x}) \cdots f_m(\bm{x}) = 0 \big\} \big|.
\end{equation*}
Thus determining $N(\bm{f}; q)$ is equivalent to counting the points of an algebraic set over $\mathbb{F}_q$, which is a much more studied problem.
However, in the case considered in this paper, working with equalities instead of inequalities does not seem more advantageous.
(Yet another way to convert from inequalities to equalities is the following: introduce auxiliary variables $y_1, \dots, y_m$ and note that $f_i \neq 0$ is equivalent to $f_i y_i = 1$ for a unique $y_i \in \mathbb{F}_q$.)
Anyway, the connection with algebraic sets suggests that there is little hope of providing a general procedure that can find a formula for $N(\bm{f}; q)$ for arbitrary polynomials $\bm{f}$.

\section{Recursive linear substitutions}\label{sec:procedure}

The procedure explained here is recursive and works as follows.

\subsection{Output format}

If successful, the procedure outputs a list of prime numbers $p_1, \dots,p_k$, a list of univariate polynomials $g_1, \dots, g_\ell \in \mathbb{Z}[u]$, and a polynomial in $k+\ell+1$ variables and integer coefficients $F$ satisfying the following property.
Let $s_i(q)$ be equal to $0$ or $1$ if the characteristic of $\mathbb{F}_q$ is equal to $p_i$ or not, respectively; and let $c_i(q)$ be equal to the number of zeros in $\mathbb{F}_q$ of the polynomial $g_i$.
Then
\begin{equation*}
    N(\bm{f}; q) = F\big(q, s_1(q), \dots, s_k(q), c_1(q), \dots, c_\ell(q)\big)
\end{equation*}
for every prime power $q$.

\subsection{Base cases}\label{subsec:base-cases}

Some base cases are easy to handle.

\begin{enumerate}[{label=(\arabic*),labelindent=1em,itemsep=0.5em,topsep=0.5em}]

    \item If one of the polynomials $f_1, \dots, f_m$ is the zero polynomial, then the procedure outputs $F(q) := 0$, since there are no solutions to the inequalities.

    \item If one of the polynomials $f_1, \dots, f_m$ is the constant polynomial $1$ or $-1$, then that polynomial can be discarded.

    \item If $m = 0$, that is, if there are no polynomial inequalities, then the procedure outputs $F(q) = q^n$, since all possible values of the variables $x_1, \dots, x_n$ have to be counted.

    \item If $m = 1$ and $f_1$ is equal to a univariate polynomial, say $f_1 = g(x)$ for some variable $x$ among $x_1, \dots, x_n$, then the procedure adds $g$ to the list of univariate polynomials of the output and outputs the polynomial $F(q, c) := (q - c) q^{n - 1}$, where $c$ is a formal variable associated to the number of zeros in $\mathbb{F}_q$ of $g$.
    Indeed, the inequality $f_1(\bm{x}) \neq 0$ holds for all possible values of the variables $x_1, \dots, x_n$ except those for which $x$ is a zero of $g$.
\end{enumerate}

\subsection{Splitting into irreducibles}

Assume that none of the base cases of Subsection~\ref{subsec:base-cases} occurs.
Then the procedure computes the pairwise distinct irreducible factors of $f_1, \dots, f_m$ over $\mathbb{Z}[\bm{x}]$.
Let these irreducible factors be the prime numbers $p_1, \dots, p_k$ and the nonconstant polynomials $h_1, \dots, h_\ell$.
It follows easily that
\begin{equation*}
    N(\bm{f}; q) = s_1(q) \cdots s_k(q) \cdot N(h_1, \dots, h_\ell; q) ,
\end{equation*}
where $s_i(q)$ is equal to $0$ or $1$ if the characteristic of $\mathbb{F}_q$ is equal to $p_i$ or not, respectively.
Thus the procedure adds $p_1, \dots, p_k$ to the list of prime numbers of the output, and it continues with the computation $N(h_1, \dots, h_\ell; q)$.

\subsection{Linear substitution}

In light of the previous considerations, assume that $f_1, \dots, f_m$ are nonconstant polynomials.
The procedure search for a polynomial $f$ among $f_1, \dots, f_m$ and a variable $x$ among $x_1, \dots, x_n$ such that $f$ is linear in $x$, that is, there exist integer-coefficient polynomials $a$ and $b$, not depending on $x$, such that $f = ax + b$.
If no such polynomial $f$ and variable $x$ are found, then the procedure aborts.
Otherwise, up to reordering, suppose that $x = x_1$ and $f = f_1$.
Also, write $\bm{y} = x_2, \dots, x_n$, and define
\begin{equation*}
    \widetilde{f}_i(\bm{y}) := f_i\!\left(-\frac{b(\bm{y})}{a(\bm{y})}, \bm{y}\right) \cdot \big(a(\bm{y})\big)^{\deg_x(f_i)}
\end{equation*}
for $i = 2, \dots, m$.
Note that $\widetilde{f}_i \in \mathbb{Z}[\bm{y}]$.
The procedure continues recursively according to the following lemma.

\begin{lemma}\label{lem:recursion}
    With the notation of this subsection, the recursive formula
    \begin{align}\label{equ:recursion}
        N(\bm{f}; q) &= N(a, f_2, \dots, f_m; q)
        + N(b, f_2, \dots, f_m; q) \\
        &\hspace{2em}- N(a, b, f_2, \dots, f_m; q)
        - N(a, \widetilde{f}_2, \dots, \widetilde{f}_m; q) \nonumber
    \end{align}
    holds.
\end{lemma}
\begin{proof}
    Let $(x, \bm{y}) \in \mathbb{F}_q^n$.
    Note that $f(x, \bm{y}) = 0$ if and only if $a(\bm{y}) = 0$ and $b(\bm{y}) = 0$, or $a(\bm{y}) \neq 0$ and $x = - b(\bm{y}) / a(\bm{y})$.
    Moreover, in the latter case, for each $i \in \{2, \dots, n\}$ the inequality $f_i(x, \bm{y}) \neq 0$ is equivalent to $\widetilde{f}_i(\bm{y}) \neq 0$.
    Therefore
    \begin{align}\label{lem:recursion:equ:1}
        N(\bm{f}; q) &= N(f_2, \dots, f_m; q) - \big|\big\{(x, \bm{y}) \in \mathbb{F}_q^n : f(x, \bm{y}) = 0 \land {\textstyle\bigwedge}_{i = 2}^m \big(f_i(x, \bm{y}) \neq 0\big) \big\}\big| \\
        &= N(f_2, \dots, f_m; q) - \big|\big\{(x, \bm{y}) \in \mathbb{F}_q^n : a(\bm{y}) = 0 \land b(\bm{y}) = 0 \land {\textstyle\bigwedge}_{i = 2}^m \big(f_i(x, \bm{y}) \neq 0\big) \big\}\big| \nonumber \\
        &\phantom{==} - N(a, \widetilde{f}_2, \dots, \widetilde{f}_m; q) . \nonumber
    \end{align}
    Furthermore, by inclusion-exclusion, it follows that
    \begin{align}\label{lem:recursion:equ:2}
        &\big|\big\{(x, \bm{y}) \in \mathbb{F}_q^n : a(\bm{y}) = 0 \land b(\bm{y}) = 0 \land {\textstyle\bigwedge}_{i = 2}^m \big(f_i(x, \bm{y}) \neq 0\big) \big\}\big| \\
        &\phantom{==}= N(f_2, \dots, f_m; q)
        - N(a, f_2, \dots, f_m; q)
        - N(b, f_2, \dots, f_m; q)
        + N(a, b, f_2, \dots, f_m; q) . \nonumber
    \end{align}
    Combining Eqs.~\eqref{lem:recursion:equ:1} and \eqref{lem:recursion:equ:2} yields Eq.~\eqref{equ:recursion}.
\end{proof}

For the sake of simplicity, the procedure aborts if recursion depth exceed a fixed parameter; so no analysis on recursion termination is needed.

\subsection{Further considerations}

This section only provides a high-level explanation, which omits several conceptually easy, but technical, optimizations/improvements.
For instance, when searching for the polynomial with a linear variable, it is better to proceed from lower- to higher-degree polynomials.
Also, during the recursive computation, the same system of inequalities may appear several times. Hence, employing memoization techniques speeds up the computation.
Furthermore, factorizing the polynomials $f_1, \dots, f_m$ into irreducible factors at each occasion might not be the most efficient strategy: a lazy approach might be faster.
Finally, if the procedure aborts, then a backtracking method makes it possible to try a different linear substitution, which perhaps leads to a successful computation.

\section{Contiguous superregular \texorpdfstring{$4 \times 4$}{4x4} matrices}\label{sec:CSR4x4}

Let $\bm{A} = (a_{i,j})$ be a contiguous superregular $4 \times 4$ matrix over $\mathbb{F}_q$.
Clearly, all entries of $\bm{A}$ are nonzero.
Also, if a row or a column of $\bm{A}$ is multiplied by a nonzero scalar, then the resulting matrix is still contiguous superregular.
Hence, starting from $\bm{A}$ and multiplying the $i$th row by $a_{i, 2}^{-1}$ for $i=1,2,3,4$ and the $j$th column by $a_{2, 2} a_{2, j}^{-1}$ for $j=1,3,4$ (in this order) produces a contiguous superregular matrix over $\mathbb{F}_q$ of the form
\begin{equation}\label{equ:normalized-matrix}
    \begin{pmatrix}
        b_{1, 1} & 1 & b_{1, 2} & b_{1, 3} \\
        1 & 1 &        1 &        1 \\
        b_{2, 1} & 1 & b_{2, 2} & b_{2, 3} \\
        b_{3, 1} & 1 & b_{3, 2} & b_{3, 3}
    \end{pmatrix} .
\end{equation}
In fact, this normalization method is a bijection.
Thus
\begin{equation}\label{equ:NCSR4x4:1}
    C_4^{\text{CSR}}(q) = (q - 1)^7 \, C_4^{\text{NCSR}}(q) ,
\end{equation}
where $C_4^{\text{NCSR}}(q)$ is the number of contiguous superregular $4 \times 4$ matrix over $\mathbb{F}_q$ of the form \eqref{equ:normalized-matrix}.

Appuswamy et al.~\cite[Lemma~4.1]{Appuswamy2026} employed a similar normalization method, but their normalized matrices have all $1$s on the first row and first column.
The normalization \eqref{equ:normalized-matrix} has the advantage that the determinants of all contiguous $2 \times 2$ submatrices are first-degree polynomials with the only exception of the determinant of the bottom-right $2 \times 2$ submatrix, which is a second-degree polynomial.

For every positive integer $k$, let $\bm{1}_k$ denote the $k \times k$ matrix with all entries equal to $1$.
If $\bm{M}$ is a $k \times k$ matrix over some field then
\begin{equation}\label{equ:1-minus-determinant}
    \left|
    \begin{array}{cccc}
        1 & 1 & \cdots & 1 \\
        1 & & & \\
        \vdots & & \bm{M} & \\[5pt]
        1 & & &
    \end{array}
    \right|
    = \left|
    \begin{array}{cccc}
        1 & 1 & \cdots & 1 \\
        0 & & & \\
        \vdots & & \bm{M} - \bm{1}_k & \\[5pt]
        0 & & &
    \end{array}
    \right|
    = (-1)^k \, |\bm{1}_k - \bm{M}| .
\end{equation}
Let $c_{i,j} := 1 - b_{i,j}$ for all $i, j \in \{1,2,3\}$.
Computing the determinants of the square contiguous submatrices of~\eqref{equ:normalized-matrix} with the help of identity~\eqref{equ:1-minus-determinant}, eventually after having swapped some rows and columns, yields that a matrix of the form~\eqref{equ:normalized-matrix} is contiguous superregular if and only if
\begin{gather}\label{equ:c-inequalities}
    c_{i,j} \neq 1 \quad \forall i,j \in \{1,2,3\} , \\
    c_{1,1} \neq 0, \quad c_{1, 2} \neq 0, \quad c_{2,1} \neq 0, \quad c_{2,2} \neq 0, \nonumber\\
    c_{1,2} \neq c_{1, 3}, \quad c_{2, 1} \neq c_{3, 1}, \quad c_{2, 2} \neq c_{2, 3}, \quad c_{2, 2} \neq c_{3, 2}, \nonumber\\
    \left|\begin{matrix}
        1 - c_{2,2} & 1 - c_{2,3} \nonumber\\
        1 - c_{3,2} & 1 - c_{3,3} \nonumber\\
    \end{matrix}\right| \neq 0 , \nonumber\\
    \left|\begin{matrix}
        c_{1,1} & c_{1,2} \\
        c_{2,1} & c_{2,2} \\
    \end{matrix}\right| \neq 0, \quad
    \left|\begin{matrix}
        c_{1,2} & c_{1,3} \\
        c_{2,2} & c_{2,3} \\
    \end{matrix}\right| \neq 0, \quad
    \left|\begin{matrix}
        c_{2,1} & c_{2,2} \\
        c_{3,1} & c_{3,2} \\
    \end{matrix}\right| \neq 0, \quad
    \left|\begin{matrix}
        c_{2,2} & c_{2,3} \\
        c_{3,2} & c_{3,3} \\
    \end{matrix}\right| \neq 0, \nonumber\\
    \left|\begin{matrix}
        c_{1,1} & c_{1,2} & c_{1,3} \\
        c_{2,2} & c_{2,2} & c_{2,3} \\
        c_{3,2} & c_{3,2} & c_{3,3} \\
    \end{matrix}\right| \neq 0 . \nonumber
\end{gather}
The degrees of inequalities \eqref{equ:c-inequalities} are still too high to applying successfully the procedure of Section~\ref{sec:procedure}.
Let
\begin{align*}
    x_{1,1} &:= c_{1,2} c_{2,1} / (c_{1,1} c_{2,2}), &
    x_{1,2} &:= 1 / c_{1,2}, &
    x_{1,3} &:= c_{1,3} / c_{1,2}, \\
    x_{2,1} &:= 1 / c_{2,1}, &
    x_{2,2} &:= 1 / c_{2,2}, &
    x_{2,3} &:= c_{2,3} / c_{2,2}, \\
    x_{3,1} &:= c_{3,1} / c_{2,1}, &
    x_{3,2} &:= c_{3,2} / c_{2,2}, &
    x_{3,3} &:= c_{3,3} / c_{2,2},
\end{align*}
so that
\begin{align*}
    c_{1,1} &:= x_{2,2} / (x_{1,1} x_{1,2} x_{2,1}), &
    c_{1,2} &:= 1 / x_{1,2}, &
    c_{1,3} &:= x_{1,3} / x_{1,2}, \\
    c_{2,1} &:= 1 / x_{2,1}, &
    c_{2,2} &:= 1 / x_{2,2}, &
    c_{2,3} &:= x_{2,3} / x_{2,2}, \\
    c_{3,1} &:= x_{3,1} / x_{2,1}, &
    c_{3,2} &:= x_{3,2} / x_{2,2}, &
    c_{3,3} &:= x_{3,3} / x_{2,2} .
\end{align*}
Hence $x_{i,j} \leftrightarrow c_{i,j}$ is a bijection over the domain on which inequalities~\eqref{equ:c-inequalities} hold.
After this change of variables, inequalities \eqref{equ:c-inequalities} are equivalent to
\begin{gather}\label{equ:final-system}
    x_{1,1} \neq 0, \quad
    x_{1,2} \neq 0, \quad
    x_{2,1} \neq 0, \quad
    x_{2,2} \neq 0, \\
    x_{1,1} \neq 1, \quad
    x_{1,2} \neq 1, \quad
    x_{2,1} \neq 1, \quad
    x_{2,2} \neq 1, \nonumber\\
    x_{1,3} \neq 1, \quad
    x_{3,1} \neq 1, \quad
    x_{2,3} \neq 1, \quad
    x_{3,2} \neq 1, \nonumber\\
    x_{1,2} - x_{1,3} \neq 0, \quad
    x_{2,1} - x_{3,1} \neq 0, \nonumber\\
    x_{1,3} - x_{2,3} \neq 0, \quad
    x_{3,1} - x_{3,2} \neq 0, \nonumber\\
    x_{2,2} - x_{2,3} \neq 0, \quad
    x_{2,2} - x_{3,2} \neq 0, \quad
    x_{2,2} - x_{3,3} \neq 0, \nonumber\\
    x_{2,3} x_{3,2} - x_{3,3} \neq 0, \nonumber\\
    x_{1,1} x_{1,2} x_{2,1} - x_{2,2} \neq 0, \nonumber\\
    (x_{2,2} - 1)(x_{2,2} - x_{3,3}) - (x_{2,2} - x_{2,3})(x_{2,2} - x_{3,2}) \neq 0, \nonumber\\
    (x_{1,1} - 1)(x_{2,3} x_{3,2} - x_{3,3}) - x_{1,1}(x_{1,3} - x_{2,3})(x_{3,1} - x_{3,2}) \neq 0 . \nonumber
\end{gather}
Running the implementation of the procedure of Section~\ref{sec:procedure} with input the system of polynomial inequalities \eqref{equ:final-system} yields as output
\begin{align}\label{equ:output}
    p_1 &= 2 , \\
    g_1 &= u^2 - u + 1 , \nonumber \\
    F &= (q - 2)(q - 3)(q^7 - 18q^6 + 143q^5 - 654q^4 + 1874q^3 - 3400q^2 + 3671q - 1855) . \nonumber
\end{align}
(For the source code, see the file \textsf{examples/NCSR\_4x4\_matrices.sage} in the repository~\cite{Repo}.)

Note that, despite the computation involving the prime number $p_1$ and the univariate polynomial $g_1$, the polynomial $F$ in the output \eqref{equ:output} depends only on $q$ (the terms involving the variables associated with $p_1$ and $g_1$ cancel out).
Hence
\begin{align}\label{equ:NCSR4x4:2}
    C_4^{\text{NCSR}}(q) &= (q - 2)(q - 3) \\
    &\phantom{1.5em} \cdot (q^7 - 18q^6 + 143q^5 - 654q^4 + 1874q^3 - 3400q^2 + 3671q - 1855) . \nonumber
\end{align}
Combining Eqs.~\eqref{equ:NCSR4x4:1} and \eqref{equ:NCSR4x4:2} yields Eq.~\eqref{equ:CSR4x4}, as desired.

\section{Contiguous superregular \texorpdfstring{$3 \times 4$, $3 \times 5$, and $3 \times 6$}{3x4, 3x5, and 3x6} matrices}\label{sec:CSR3x4-5-6}

For all positive integers $k$ and $\ell$, let $C_{k \times l}^{\textsf{CSR}}(q)$ denote the number of contiguous superregular $k \times \ell$ matrices over $\mathbb{F}_q$.
After the same normalization explained in Section~\ref{equ:CSR4x4}, that is, normalizing the matrix so that the second row and the second column have all entries equal to $1$, a direct application of the procedure of Section~\ref{sec:procedure} provides the following formulas.

\begin{samepage}
\begin{align*}
    C_{3 \times 4}^{\textsf{CSR}}(q) &= (q - 1)^6 (q - 2) (q - 3) (q^4 - 9q^3 + 32q^2 - 54q + 37) , \\
    C_{3 \times 5}^{\textsf{CSR}}(q) &= (q - 1)^7 (q - 2) (q - 3)^2 (q^3 - 6q^2 + 14q - 13)(q^2 - 5q + 7) , \\
    C_{3 \times 6}^{\textsf{CSR}}(q) &= (q - 1)^8 (q - 2) (q - 3)^2 \\
    &\hspace{2em}\cdot (q^7 - 16q^6 + 113q^5 - 458q^4 + 1154q^3 - 1813q^2 + 1649q - 671) .
\end{align*}
\end{samepage}
(For the source code, see the files \textsf{examples/NCSR\_3x*\_matrices.sage} in the repository~\cite{Repo}.)

\bibliographystyle{amsplain-no-bysame}
\bibliography{temp}

\end{document}